\documentclass[12pt,a4paper,reqno]{article}
\usepackage[T1]{fontenc}
\usepackage{authblk}
\pdfoutput=1
\usepackage{amsmath}
\usepackage{amsfonts,amsthm,amssymb}
\usepackage{mathabx}
\usepackage{bm}
\usepackage{amscd}
\usepackage[latin2]{inputenc}
\usepackage{t1enc}
\usepackage[mathscr]{eucal}
\usepackage{indentfirst}
\usepackage{graphicx}
\usepackage{graphics}
\usepackage{pict2e}
\usepackage{epic}
\numberwithin{equation}{section}
\usepackage[margin=3cm]{geometry}
\usepackage{epstopdf}
\usepackage[colorlinks,linkcolor=blue]{hyperref}
\usepackage[capitalise,noabbrev]{cleveref}
\usepackage{todonotes}
\usepackage{verbatim}

\crefformat{equation}{(#2#1#3)}
\crefmultiformat{equation}{(#2#1#3)}{ and~(#2#1#3)}{, (#2#1#3)}{ and~(#2#1#3)}
\crefrangeformat{equation}{(#3#1#4) to~(#5#2#6)}

\usepackage[normalem]{ulem}
\usepackage[misc,geometry]{ifsym}

\allowdisplaybreaks

\theoremstyle{plain}
\newtheorem{Thm}{Theorem}[section]
\newtheorem*{Thm*}{Theorem}
\newtheorem{Lem}[Thm]{Lemma}

\newtheorem{Prop}[Thm]{Proposition}

\theoremstyle{definition}

\newtheorem{Rem}[Thm]{Remark}
\newtheorem{?}[Thm]{Problem}

\newcommand{\R}{\mathbb{R}}

\newcommand{\blue}{\textcolor{blue}}

\begin{document}

\title{
Asymptotic stability of the combination of a viscous contact wave with two rarefaction waves for 1-D Navier-Stokes equations under periodic perturbations }

	\author{ Lingjun Liu\thanks{College of Mathematics, Faculty of Science, Beijing University of Technology, Beijing 100124, P.R.China. 
	E-mail: lingjunliu@bjut.edu.cn.} ,
   {Danli Wang\textsuperscript{\Letter}} \thanks{Corresponding author. 
    Academy of Mathematics and Systems Science, Chinese Academy of Sciences, Beijing 100190, P.R.China. School of Mathematical Sciences, University of Chinese Academy of Sciences, Beijing 100049, P.R.China.
	E-mail: wangdanli19@amss.ac.cn.} ,
    Lingda Xu\thanks{Department of Mathematics, Yau Mathematical Sciences Center, Tsinghua University, Beijing 100084, P.R.China. Yanqi Lake Beijing Institute of Mathematical Sciences And Applications, Beijing 101408, P.R.China.
	E-mail: xulingda@tsinghua.edu.cn.}}
	\date{}

\maketitle
\begin{abstract}
Considering the space-periodic perturbations, we prove the time-asymptotic stability of the composite wave of a viscous contact wave and two rarefaction waves for the Cauchy problem of 1-D compressible Navier-Stokes equations in this paper. This kind of perturbations keep oscillating at the far field and are not integrable. The key is to construct a suitable ansatz carrying the same oscillation 
of the solution as in \cite{HuangXuYuan2020,HuangYuan2021}, but due to the degeneration of contact discontinuity, the construction is more subtle. We find a way to use the same weight function for different variables and wave patterns, which still ensure the errors be controllable. Thus, this construction can be applied to contact discontinuity and composite waves. Finally, by the energy method, we prove that the Cauchy problem admits a unique global-in-time solution and the composite wave is still stable under the space-periodic perturbations.

\end{abstract}
\section{Introduction}

We study the one-dimensional (1-D) full compressible Navier-Stokes equations in the Lagrangian coordinates: 
\begin{equation}\label{Equ1}
\begin{cases}
v_{t}-u_{x}=0,\\
u_{t}+p_{x}={\mu}(\frac{u_{x}}{v})_{x},  \qquad  \qquad \qquad \ \ \ \ \ \ \ \ \ \ \ \ x\in\R,\  t>0, \\
E
_{t}+(pu)_{x}={\kappa}(\frac{{\theta}_{x}}{v})_{x}+{\mu}(\frac{u{u_{x}}}{v})_{x},
\end{cases}
\end{equation}
where $v(x,t)>0, u(x,t)\in\R$ and $\theta(x,t)>0$ are the specific volume, fluid velocity and absolute temperature, respectively. $\mu >0$ is the viscosity coefficient and $\kappa >0$ denotes the coefficient of heat conductivity. 
The pressure $p(x,t)$ and the 
total energy $E=E(u,\theta)$
satisfy
 $$p=\frac{R\theta}{v}=Av^{-\gamma}e^{\frac{\gamma-1}{R}s},\ E(u,\theta)=e(\theta)+{\frac{u^2}{2}}
 =\frac{R\theta}{\gamma-1}+{\frac{u^2}{2}}
 ,$$
 where $\gamma>1$ is the adiabatic exponent, $e$ is the internal energy, $s$ denotes the entropy and $A>0, R>0$ are fluid constants. 


When $\mu =\kappa =0$, the system \eqref{Equ1} is formally reduced to the following compressible Euler system
\begin{equation}\label{Equ2}
\begin{cases}
v_{t}-u_{x}=0,\\
u_{t}+p_{x}=0, \\ 
E
_{t}+(pu)_{x}=0.
\end{cases}
\end{equation}
The system 
\eqref{Equ2} is a typical strictly hyperbolic conservation laws with 
$$\lambda_1=-\sqrt{(\frac{\gamma p}{v})}<0, \ \  \lambda_2=0, \ \ \lambda_3=-\lambda_1>0.$$
It's known that the Riemann solutions of \eqref{Equ2} with the initial Riemann data 
 \begin{align}\label{Ini}
 	(v, u, \theta)(x,0)=
	\begin{cases}
	(v_-, u_-, \theta_-),\ \  x<0,\\
 (v_+, u_+, \theta_+),	\ \ x>0,
\end{cases}
\end{align}
consist of contact discontinuities, 
shock waves, 
rarefaction waves and the 
 linear superpositions of these three basic wave patterns 
(see \cite{Smoller}). And the viscous version of the corresponding Riemann solutions characterize the large-time behaviors of solutions to \eqref{Equ1} when the initial data tend to different constant states as $x\
\to \pm \infty$.



There are extensive literatures on the stability analysis of the viscous wave pattern to system \eqref{Equ1}, such as \cite{MatsumuraNishihara1985, Goodman1986}
for the shock waves, \cite{Matsumura1986,KawashimaMatsumuraNishihara1986,NishiharaYangZhao2004} for the rarefaction waves, \cite{Huang2004,Huang2006,Huang2008,LiuXin1997,Xin1996} for the viscous contact wave. 
However, the analysis of the stability of different basic wave patterns rely on their underlying properties and the framework of which are not compatible with each other, which makes it challenging to establish the stability of the superposition of several wave patterns for the Navier-Stokes equations.
The nonlinear stability of the superposition of two viscous shocks for the full compressible Navier-Stokes equations was achieved by Huang-Matsumura \cite{Huang2009}. Huang-Li-Matsumura \cite{Huang2010} proved the asymptotic stability of the combination wave of viscous contact wave with rarefaction waves for 1-D compressible Navier-Stokes equations by using the elementary energy methods  with the aid of a key inequality for the heat kernel. Then Huang-Wang \cite{Huang2016} showed the stability of this kind of combination wave under large initial perturbation for the Navier-Stokes equations without restriction on the adiabatic exponent $\gamma$. 
Recently, by utilizing the theory of $a$-contraction with shifts, Kang-Vasseur-Wang \cite{Wangyi2021} investigated the time-asymptotic stability for the composite waves
of a viscous shock and a rarefaction wave for the 1-D compressible barotropic Navier-Stokes equations
. More results on combination waves for other significant systems refer to \cite{LiWangWang2018, RuanYinZhu2017, YaoZhu2021} 
 and the references therein.  

It is interesting and important to study the stability of wave patterns under space-periodic perturbations for the hyperbolic conservation laws. Different from the perturbations in previous results, the space-periodic perturbations have infinite oscillations at the far field and are not integrable in $\mathbb{R}$. Lax and Glimm \cite{Lax1957,Glimm1970} established the global-in-time existence of periodic solutions, which asymptotically tends to their constant averages at rates $t^{-1}$. \cite{XinYuanYuan2019, XinYuanyuanIndiana, YuanYuan2019} studied the asymptotic stability of shocks and rarefaction waves under the periodic perturbations for the 1-D scalar conservation laws in both inviscid and viscous case. Recently, Huang-Xu-Yuan \cite{HuangXuYuan2020} shown the asymptotic stability of planar rarefaction waves under space-periodic perturbation for multi-dimensional Navier-Stokes equations by constructing a suitable ansatz, which makes the energy method still avaibled. Furthermore, \cite{HuangYuan2021} and \cite{YuanYuan2021} studied the stability of a single viscous shock and a composite wave of two viscous shocks under space-periodic perturbations for the 1-D Navier-Stokes equations, respectively. 

In this paper, we study the asymptotic stability of a composite  wave of viscous contact wave and the rarefaction waves under general periodic perturbations for the 1-D compressible Navier-Stokes equations \eqref{Equ1}. Motivated by \cite{HuangXuYuan2020},
we want to construct a suitable ansatz and then use the basic energy method to establish the stability result. However, the ansatz in \cite{HuangXuYuan2020}
can not be applied to the viscous contact wave in our paper since the velocity of which does not have different end states, which is crucial for the construction of the ansatz. Indeed, the aim of constructing the suitable ansatz is that the ansatz shares 
the same frequencies of oscillations as those 
of the solution at the far field, so that their initial difference is integrable which enables us to use the basic energy method.
Therefore, to overcome the difficulty comes from the viscous contact wave we mentioned above, we choose only one appropriate smooth function to construct the ansatz for $v,\ u,\ \theta$ and different wave patterns.
After this, we can use the framework as in \cite{Huang2010} to establish the stability of a composite  wave of viscous contact wave and rarefaction waves under general periodic perturbations.

This paper is organized as follows. In Section 2, some wave profiles are introduced and the main result Theorem \ref{thm1} is presented.
Some useful lemmas are given in section 3.
Section 4 is devoted to the proof of the main result Theorem \ref{thm1}.

\section{Wave profiles and main result}

At the beginning, we introduce wave profiles of contact discontinuity,
rarefaction waves and the composite wave. In the second part, we present the main result. 
\subsection{Wave profiles}
{\bf {Contact discontinuity.}}
If $(v_-,u_-,\theta_-)\in CD(v_+,u_+,\theta_+)$, which is the contact discontinuity wave curve:
$$CD(v_+,u_+,\theta_+):=\Big\{(v,u,\theta)\big|u=u_+, p=p_+, v \not\equiv v_+ \Big\},$$
then the contact discontinuity solution of \eqref{Equ2} with \eqref{Ini}
is
 \begin{align}\label{CDs}
 	(\hat V, \hat U, \hat\Theta)(x,t)=
	\begin{cases}
	(v_-, u_-, \theta_-),\ \  x<0, \ t>0,\\
 (v_+, u_+, \theta_+),	\ \ x>0 \ t>0,
\end{cases}
\end{align}
with
\begin{align}\label{equ3}
	u_-=u_+, \ \ \  p_-:=\frac{R\theta_-}{v_-}=p_+:=\frac{R\theta_+}{v_+}.
\end{align}

Here we consider the viscous contact wave for the corresponding compressible Navier-Stokes equations \eqref{Equ1}.
We first construct the viscous contact wave $(V, U, \Theta)$ for the Navier-Stokes equations \eqref{Equ1}, which is introduced in \cite{Huang2008} and \cite{Huang2010}. 
We start from the following nonlinear diffusion equation,
\begin{align}\label{equ5}
\begin{cases}
\Theta_t=a(\frac{\Theta_x}{\Theta})_x, \ \ a=\frac{\kappa p_+(\gamma-1)}{\gamma R^2}>0,\\
\Theta(\pm\infty, t)=\theta_{\pm}.	
	\end{cases}
\end{align}
It is easy to know that $\Theta(x,t)=\Theta(\xi), \xi=\frac{x}{\sqrt{1+t}}$ is a unique self-similar solution to \eqref{equ5} (see \cite{Hsiao1993, Van1976}). And $\Theta(\xi)$ is a monotone function, increasing (or decreasing) if $\theta_+>\theta_-$ (or $\theta_+<\theta_-$). Furthermore, $\exists \bar\delta>0$, 
such that for $\delta=|\theta_+-\theta_-|\le\bar\delta$,  there holds
\begin{align}\label{equ6}
\begin{aligned}
(1+t)^{\frac{3}{2}}|\Theta_{xxx}|+&(1+t)|\Theta_{xx}|+(1+t)^{\frac{1}{2}}|\Theta_x|+|\Theta-\theta_\pm|\\
&\le C_1\delta e^{-\frac{C_2x^2}{1+t}}\ \  \  \text{as}\ \ \ x\rightarrow\pm\infty,
\end{aligned}
\end{align}
where positive constants $C_1$ and $C_2$ only depend on $\theta_-$ and $\bar\delta$. Then the contact wave proflie $(V, U, \Theta)(x,t)$ is
\begin{align}\label{equ7}
V=\frac{R}{p_+}\Theta, \ \ U=u_-+\frac{\kappa(\gamma-1)}{\gamma R}\frac{\Theta_x}{\Theta}, \ \ \Theta=\Theta.
\end{align}
By direct calculation, we find that
\begin{align}\label{ap}
\|V-\hat V, U-\hat U, \Theta-\hat\Theta\|_{L^p}=O(\kappa^{\frac{1}{2p}})(1+t)^{\frac{1}{2p}} \ \ \text{as}\ \ \kappa\rightarrow 0.
\end{align}
The contact wave $(V, U, \Theta)$ constructed in \eqref{equ7} satisfies the system
\begin{equation}\label{equ8}
\begin{cases}
V_{t}-U_{x}=0,\\
U_{t}+(\frac{R\Theta}{V})_{x}=\mu(\frac{U_x}{V})_x+Q_1, \\ 
E_{t}(V,\Theta)
+(p(V,\Theta)U)_{x}=(\kappa\frac{\Theta_x}{V}+\mu\frac{UU_x}{V})_x+Q_2,
\end{cases}
\end{equation}
where as $|x|\rightarrow\infty,$
\begin{align}\label{Q1}
\qquad\ \ \  Q_1&=\frac{\kappa(\gamma-1)}{\gamma R}\Big((\ln\Theta)_{xt}-\mu\big(\frac{p_+}{R\Theta}(\ln\Theta)_{xx}\big)_x\Big)=O(\delta)(1+t)^{-\frac{3}{2}}e^{-\frac{C_2x^2}{1+t}},\\\label{Q2}
Q_2&=\Big(\frac{\kappa(\gamma-1)}{\gamma R}\Big)^2\Big((\ln\Theta)_{x}(\ln\Theta)_{xt}-\mu\big(\frac{p_+}{R\Theta}(\ln\Theta)_{x}(\ln\Theta)_{xx}\big)_x\Big)\\
&=O(\delta)(1+t)^{-{2}}e^{-\frac{C_2x^2}{1+t}}.\nonumber
\end{align}

{\bf{ Rarefaction waves.}} If $(v_-,u_-,\theta_-)\in R_i(v_+,u_+,\theta_+)$ $(i=1,3)$,
$$R_i(v_+,u_+,\theta_+):=\Big\{(v,u,\theta)\big|v<v_+,u=u_+-\int_{v_+}^v\lambda_i(\eta,s_+)d\eta,s(v,\theta)=s_+\Big\},$$
with $$s=\frac{R}{\gamma-1}\ln{\frac{R\theta}{A}}+R\ln v,\ \ s_\pm=\frac{R}{\gamma-1}\ln{\frac{R\theta_\pm}{A}}+R\ln v_\pm,$$
then there exists an i-rarefaction wave $(v^{R_i}, u^{R_i}, \theta^{R_i})(\frac{x}{t})$ which   
is the weak solution of the Euler system \eqref{Equ2} with the initial Riemann data
 \begin{align}\label{IR}
 	(v,u,\theta)(x,0)=
	\begin{cases}
	(v_-,u_-,\theta_-),
	\ \  
	x<0,\\
 (v_+,u_+,\theta_+),
 \ \ 
 x>0. 
\end{cases}
\end{align}
Since the rarefaction wave is only Lipschitz continuous, motivated by \cite{Matsumura1986}, we use the smooth solutions $(V^{R_i}, U^{R_i}, \Theta^{R_i})(i=1,3)$ of \eqref{Equ2} to approximate the rarefaction waves $(v^{R_i}, u^{R_i}, \theta^{R_i})(\frac{x}{t})$. 
 And $(V^{R_i}, U^{R_i}, \Theta^{R_i})$ are given by the following form
 \begin{equation}\label{equ9}
	\begin{cases}
	S^{R_i}(x,t)=s\big(V^{R_i}(x,t), \Theta^{R_i}(x,t)\big)=s_+, \\
	w_\pm=\lambda_{i\pm}:=\lambda_i(v_\pm,\theta_\pm),\\
	\lambda_i(V^{R_i}(x,t),s_+)=w(x,t),\\
	U^{R_i}=u_+-\int_{v_+}^{V^{R_i}(x,t)}\lambda_i(\eta,s_+)d\eta,\\
\end{cases}
\end{equation}
where $w$ 
is the solution of inviscid Burgers equation 
 \begin{align}\label{equ10}
	\begin{cases}
	w_t+ww_x=0,\\
	w(x,0)=\frac{1}{2}(w_++w_-)+(w_+-w_-)\frac{\tanh x}{2}.
\end{cases}
\end{align}

{\bf{  Superposition of rarefaction waves and viscous contact wave.}}
Similar to \cite{Huang2010}, we assume that
$$(v_-,u_-,\theta_-)\in \Omega_{R_1-CD-R_3}(v_+,u_+,\theta_+).$$
From \cite{Smoller},
$\exists \delta_1>0$ suitably small such that for
\begin{align}\label{delta1}
(v_-, u_-, \theta_-)\in \Omega_{R_1-CD-R_3}(v_+,u_+,\theta_+), \ \ |\theta_--\theta_+|\le \delta_1,
\end{align}
there exists uniquely a pair of points $(v_-^o, u^o, \theta_-^o)\ \text{and}\ (v_+^o, u^o, \theta_+^o)
$ such that
$$(v_-,u_-,\theta_-)\in R_1(v^o_-,u_-^o,\theta_-^o), \ (v^o_-,u_-^o,\theta_-^o) \in CD(v^o_+,u_+^o,\theta_+^o),$$ $$(v^o_+,u_+^o,\theta_+^o)\in R_3(v_+,u_+,\theta_+),$$
and
\begin{align}\label{21}
|v^o_\pm-v_\pm|
+|u^o-u_\pm|+|\theta^o_\pm-\theta_\pm|\le C_3|\theta_--\theta_+|,
\end{align}
for some positive constant $C_3$. 

Thus the Riemann solution $(\bar V, \bar U, \bar \Theta)(t,x)$ of the Euler system \eqref{Equ2}, consisting of two rarefaction waves and a contact discontinuity, can be defined by
\begin{equation}\label{riemann}
\left(
	\begin{array}{c}
	\bar V\\
	\bar U\\
	\bar \Theta
	\end{array}
\right)(x,t)=\left(
	\begin{array}{cc}
	v^{R_1}+\hat V
	+v^{R_3}\\
	u^{R_1}+\hat U
	+u^{R_3}\\
	\theta^{R_1}+\hat \Theta
	+\theta^{R_3}
	\end{array}
\right)(x,t)-\left(
	\begin{array}{ccc}
	v_-^{o}+v_+^o\\
	u^o+u^o\\
	\theta_-^o+\theta_+^o
	\end{array}
\right)(x,t).
\end{equation}
Corresponding to \eqref{riemann}, we now can define the approximate wave pattern. For convenience, we denote $(v^{cd},u^{cd}, \theta^{cd})(x,t)$ as the viscous contact wave from \eqref{equ5} and \eqref{equ7} with the constant states $(v_\pm, u_\pm, \theta_\pm)$ replaced by $(v^o_\pm, 0, \theta_\pm^o)$, and rewrite $$(v_{r_1},u_{r_1}, \theta_{r_1}):=(V^{R_1}, U^{R_1}, \Theta^{R_1}), \ \  (v_{r_3},u_{r_3}, \theta_{r_3}):=(V^{R_3}, U^{R_3}, \Theta^{R_3}).$$ 
Then we set
\begin{equation}\label{equ11}
\left(
	\begin{array}{c}
	\widetilde V\\
	\widetilde U\\
	\widetilde \Theta
	\end{array}
\right)(x,t)=\left(
	\begin{array}{cc}
v_{r_1}+v^{cd}+v_{r_3}\\
u_{r_1}+u^{cd}+u_{r_3}\\
\theta_{r_1}+\theta^{cd}+\theta_{r_3}
	\end{array}
\right)(x,t)-\left(
	\begin{array}{ccc}
	v_-^{o}+v_+^o\\
	u^o\\
	\theta_-^o+\theta_+^o
	\end{array}
\right)(x,t).
\end{equation}

\subsection{Main result}
In this paper, we consider the stability of composite wave of a viscous contact wave and two rarefaction waves for 1-D compressible Navier-Stokes equations \eqref{Equ1} under periodic perturbation, i.e., the periodic initial data satisfies
 \begin{align}\label{pIB}
	(v_0, u_0, E
	_0)(x)=(\widetilde V,\widetilde U,\widetilde E
	)(x,0)+(\phi_{1}, \phi_{2}, \phi_{3})(x),
\end{align}
where $E_0=e(\theta_0)+{\frac{u_0^2}{2}}=\frac{R}{\gamma-1}\theta_0+\frac{1}{2}u_0^2$, 
$\widetilde E=\frac{R}{\gamma-1}\widetilde\Theta+\frac{1}{2}\widetilde U^2\ $
and $\ \phi_{ j}\in H^3(0,\pi^*)(j=1,2,3)$ are periodic perturbations with period $\pi^*>0 $ satisfying
\begin{align}\label{zero}
	\int_0^{\pi^*}(\phi_{1}, \phi_{2}, \phi_{3})(x)dx=0.
\end{align}
We set $\pi^*=1$ for simplicity.

\begin{Rem}\label{rem1}
The periodic perturbation should be imposed on $(\widetilde V,\widetilde U,\widetilde E)(x,0)$, which are conservative quantities. This is crucial for the exponential decay rates of the periodic solutions, which are shown in Lemma \ref{lemm1} below.
\end{Rem}


Then we rewrite the initial data $(v_0, u_0, \theta_0)(x)$ 
as follows:
\begin{align}\label{eqin}
\begin{aligned}
&(v_0,u_0,\theta_0)(x)=\Big(\widetilde V(x,0)+\phi_1(x),\widetilde U(x,0)+\phi_2(x), \widetilde\Theta(x,0)+\phi_4(x)\Big), 
\end{aligned}
\end{align}
 where $$\phi_4(x)=\frac{\gamma-1}{2R}\big[\widetilde U^2(x,0)-\big(\widetilde U(x,0)+\phi_2(x)\big)^2\big]+\frac{\gamma-1}{R}\phi_3(x).$$

Let $(\bar v_\pm, \bar u_\pm, \bar 
\theta
_\pm)$ and $(\bar v^o_\pm, \bar u_\pm^o,
 \bar 
 \theta
 ^o_\pm)$ denote the periodic solutions of the Navier-Stokes equations \eqref{Equ1} with the periodic initial datas $(\bar v_\pm, \bar u_\pm, \bar E
 _\pm)(x,0)$ and $(\bar v^o_\pm, \bar u_\pm^o, \bar E
 ^o_\pm)(x,0),$ respectively 
  (see \cite{YuanYuan2021} for the existence): 
 \begin{align}\label{p1}
 \begin{aligned}
 (\phi_1,\phi_2,\phi_3
 )(x)&=(\bar v_\pm,\bar u_\pm, \bar E
 _\pm)(x,0)-(v_\pm, u_\pm, E
 _\pm)\\
 &=(\bar v^o_\pm, \bar u_\pm^o, \bar E
 ^o_\pm )(x,0)-(v^o_\pm, u^o, E
 ^o_\pm),
 \end{aligned}
 \end{align}
 and denote
\begin{align}\label{equ21}
\begin{array}{c}
(\widetilde v_{\pm}, \widetilde u_\pm, \widetilde  E
_\pm)(x,t)=(\bar v_\pm-v_\pm, \bar u_\pm-u_\pm, \bar E
_\pm-E
_\pm)(x,t), \\
(\widetilde v^o_{\pm}, \widetilde u^o_\pm, \widetilde E
^o_\pm)(x,t)=(\bar v^o_\pm-v^o_\pm, \bar u_\pm^o
-u^o, \bar E
^o_\pm-E
^o_\pm)(x,t),
\end{array}
\end{align}
where
\begin{align}
&E_\pm=E(u_\pm,\theta_\pm),\ E^o_\pm=E(u_\pm^o,\theta_\pm^o),\\
&\bar E_\pm=E(\bar u_\pm,\bar\theta_\pm),\ \bar E^o_\pm=E(\bar u_\pm^o,\bar\theta^o_\pm).
\nonumber
\end{align}

Set
\begin{equation}\label{equ14}
\eta=\frac{v^{cd}-v^o_-}{v^o_+-v^o_-}.
\end{equation}
Now we are ready to construct the ansatz. 
  We define

\begin{equation}\label{equ15}
\begin{array}{l}
(\bar v_{r_1}, \bar u_{r_1}, \bar\theta_{r_1})\\
\ =\Big((1-\eta)\widetilde v_-+\eta\widetilde v^o_- +v_{r_1},(1-\eta)\widetilde u_-+\eta\widetilde u^o_- +u_{r_1},(1-\eta)\widetilde \theta_-+\eta\widetilde \theta^o_- +\theta_{r_1} \Big),\\
	(\bar v_{r_3}, \bar u_{r_3}, \bar\theta_{r_3})\\
	\ =\Big((1-\eta)\widetilde v_+^o+\eta\widetilde v_+ +v_{r_3},  (1-\eta)\widetilde u_+^o+\eta\widetilde u_+ +u_{r_3},  (1-\eta)\widetilde \theta_+^o+\eta\widetilde \theta_+ +\theta_{r_3} \Big) ,\\
	(\bar v^{cd}, \bar u^{cd}, \bar\theta^{cd})\\
	\quad=\Big((1-\eta)\widetilde v^o_-+\eta\widetilde v^o_+ +v^{cd}, u^{cd}+(1-\eta)\widetilde u^o_-+\eta\widetilde u^o_+ +u^o,(1-\eta)\widetilde \theta^o_-+\eta\widetilde \theta^o_+ +\theta^{cd} \Big),
\end{array}
\end{equation}
where
$\widetilde \theta_{\pm}=\bar \theta_\pm-\theta_\pm,\ \widetilde \theta^o_{\pm}=\bar \theta^o_\pm-\theta^o_\pm$.
Then the ansatz $(\bar v, \bar u, \bar\theta)$ is
\begin{equation}\label{equ16}
\left(
	\begin{array}{c}
	\bar v\\
	\bar u\\
	\bar\theta
	\end{array}
\right)(x,t)=\left(
	\begin{array}{cc}
	\bar v^{cd}+\bar v_{r_1}+\bar v_{r_3}\\
	\bar u^{cd}+\bar u_{r_1}+\bar u_{r_3}\\
	\bar \theta^{cd}+\bar\theta_{r_1}+\bar\theta_{r_3}
	\end{array}
\right)(x,t)-\left(
	\begin{array}{ccc}
	\bar v_-^{o}+\bar v_+^o\\
	\bar u_-^o+\bar u_+^o\\
	\bar\theta_-^o+\bar\theta_+^o
	\end{array}
\right)(x,t).
\end{equation}

Set
\begin{align}\label{19}
\begin{aligned}
	&(\phi, \psi,
	\zeta)(x,t)=(v-\bar v, u-\bar u, 
	 \theta-\bar\theta)(x,t),\\
	  &(\phi,\psi,
	   \zeta)(x,
	0):=(\phi_0,\psi_0,
	\zeta_0)(x),\ \ w(x,0):=w_0=E_0-\bar E_0,
\end{aligned}
\end{align}
where $\bar E_0
=\frac{R}{\gamma-1}\bar\theta_0+\frac{1}{2}\bar u_0^2$ ,
and $(v, u,
 \theta)$ is the solution of Navier-Stokes equations \eqref{Equ1}. For any $T>0$, 
we define 
\begin{align*}
X(I)\subset C\big(I;H^1(\mathbb{R})\big),\ \text{where}\   I:=[0,T]\subset[0,+\infty),
\end{align*}
where
\begin{align*}
X(I)=\Big\{(\phi,\psi,\zeta)\in C\big(I;H^1(\mathbb{R})\big)\Big| \phi_x\in L^2(I;L^2\big(\mathbb{R})\big), (\psi_x,\zeta_x)\in L^2\big(I; H^1(\mathbb{R})\big)
\Big\}.
\end{align*}
 In this paper, we denote $C$ as the positive generic constant and the norm of the $l-$th order Sobolev space 
 $H^l(\mathbb{R})$ as $$\|f\|_l=\big(\sum_{j=0}^{l}\|\partial_x^jf\|^2\big)^{\frac{1}{2}},\ \text{where}\  \|\cdot\|:=\|\cdot\|_{L^2(\mathbb{R})}.$$

 Now we 
 state the main result about the compressible Navier-Stokes equations \eqref{Equ1}:
\begin{Thm}\label{thm1}
Assume that the periodic perturbations $(\phi_1,\phi_2,\phi_3) \in H^3((0,1))$ satisfy \eqref{zero}.
 Given a Riemann solution $(\bar V, \bar U, \bar\Theta)$ defined in \eqref{riemann}, there exist positive constants $\delta_0\big(\le\min\big\{\bar\delta, \delta_1,1\big\}\big)$ and $\epsilon_0>0$, 
  such that if $$|\theta_--\theta_+|=\delta\le\delta_0,\ \ 
 \|(\phi_{1}, \phi_{2}, \phi_{3})(x,0)\|_{H^3((0,1))}\le \epsilon_0,
 $$ then the compressible Navier-Stokes equations \eqref{Equ1} with \eqref{pIB} admits a unique global solution satisfying that
\begin{align}\label{eq020}
\begin{aligned}
(v-\bar v, u-\bar u,\theta-\bar \theta)\in  X(I).
\end{aligned}
\end{align}
Moreover, it holds that
\begin{equation}\label{eq030}
\lim_{t\rightarrow\infty}\sup_{x\in\mathbb{R}}|(v,u,\theta)(x,t)-(\bar V,\bar U,\bar \Theta)(x,t)|=0.
\end{equation}
\end{Thm}

\begin{Rem}\label{rem2}
Motivated by \cite{HuangXuYuan2020}, we try to construct a suitable ansatz $(\bar v, \bar u, \bar\theta)$ to eliminate the oscillations at the far field, induced by the periodic perturbations. However, the ansatz introduced in \cite{HuangXuYuan2020} can not be used for viscous contact wave in our case. In fact, we can not construct a smooth function $\frac{u^{cd}-u^{o}_-}{u^{o}_+ - u^{o}_-}$ like that in \cite{HuangXuYuan2020} since for viscous contact wave, $u^{o}_+ = u^{o}_- $. We observe that, only one smooth function $\eta=\frac{v^{cd}-v^o_-}{v^o_+-v^o_-}$ is needed to construct the ansatz for $v,\ u,\ \theta$ and different wave patterns (see \eqref{equ16}). It turns out that our ansatz still carry the same frequencies of oscillations of the solution so that the initial perturbations $(v-\bar v, u-\bar u,\theta-\bar \theta)(x,0)$ become integrable. Additionally, we use the smooth function $\eta=\frac{v^{cd}-v^o_-}{v^o_+-v^o_-}$ just for clarity, it can also be replaced by another similar functions such as $\frac{\theta_{r_1}-\theta_-}{\theta^o_--\theta_-}$ and so on.
\end{Rem}

\section{Preliminaries}
In this section, we first show
some useful properties of  viscous contact wave and rarefaction waves.
Then we 
 present the properties of periodic solutions to \eqref{Equ1}.

Let $\mathbb{R}\times (0,t)=\Omega_-\cup\Omega_c\cup\Omega_+$, where
\begin{equation}
\begin{aligned}
&\Omega_+=\{(x,t)|2x > \lambda_3(v_+^o, s_+)t\},\\
&\Omega_-=\{(x,t)|-2x > - \lambda_1(v_-^o, s_-)t\},\\
&\Omega_c=\{(x,t)|\lambda_1
(v^o_-, s_-)t\le2x\le\lambda_3
(v^o_+, s_+)t\}.\nonumber
\end{aligned}
\end{equation}
There are some basic properties for the viscous contact wave, rarefaction waves as follows.

\begin{Lem}\label{lemm2}(\cite{Huang2010})
Assume that $(v_-, u_-, \theta_-)$ satisfies \eqref{delta1} with $\delta:=|\theta_--\theta_+|\le\bar\delta$ for any given $(v_+, u_+,\theta_+)$, then the smooth i-rarefaction waves $(v_{r_i},u_{r_i},\theta_{r_i})\ (i=1,3)
$ constructed in \eqref{equ9} and the viscous contact discontinuity wave $(v^{cd}, u^{cd}, \theta^{cd})$ satisfy:\\

(i) $(u_{r_i})_x\ge 0\ (x\in\mathbb{R}, t>0)$.

(ii) For $1\le p\le\infty,$ $t\ge0$ and $\delta=|\theta_--\theta_+|$, there exists a positive constant $C=C(p,v_-,u_-,\theta_-,\bar\delta, \delta_1)$ such that
\begin{align}\label{eq22}
\|\big((v_{r_i})_x, (u_{r_i})_x, (\theta_{r_i})_x\big)(t)\|_{L^p}
\le C\min\big\{\delta,\delta^{\frac{1}{p}}t^{-1+\frac{1}{p}}\big\},
\end{align}
and
\begin{align}\label{eq23}
\sum_{l=2,3}\|\partial_x^l(v_{r_i}, u_{r_i},\theta_{r_i})(t)\|_{L^p}
\le C\min\big\{\delta,t^{-1}\big\}.
\end{align}

(iii) There exists some positive constant $C=C(v_-,u_-,\theta_-,\bar\delta,\delta_1)$ such that
\begin{align}\label{eq24}
\begin{aligned}
(u_{r_i})_x&+|(v_{r_i})_x|+|v_{r_i}-v^o_\pm|+|(\theta_{r_i})_x|+|\theta_{r_i}-\theta_{r_i}^o|\\
&+|(u_{r_i})_{xx}|+|(v_{r_i})_{xx}|+|(\theta_{r_i})_{xx}|\le C\delta e^{-c_0(|x|+t)},
\end{aligned}
\end{align}
in $\Omega_c$, and
\begin{align}\label{25}
\begin{cases}
&|v^{cd}-v^o_\mp|+|(v^{cd})_x|+|\theta^{cd}-\theta_\mp^o|+|(\theta^{cd})_x|+|(u^{cd})_{x}|\le C\delta e^{-c_0(|x|+t)}, \\
&(u_{r_i})_x+|(v_{r_i})_x|+|v_{r_i}-v^o_\pm|+|(\theta_{r_i})_x|+|\theta_{r_i}-\theta_\pm^o|\\
&\qquad+|(u_{r_i})_{xx}|+|(v_{r_i})_{xx}|+|(\theta_{r_i})_{xx}|\le C\delta e^{-c_0(|x|+t)},
\end{cases}
\end{align}
in $\Omega_{\mp}$,
where $$c_0=\frac{1}{10}\min\big\{|\lambda_1(v_-^o, s_-)|, \lambda_3(v_+^o, s_+), C_2\lambda^2_1(v_-^o, s_-), C_2\lambda^2_3(v_+^o, s_+), 1\big\}.$$

(iv) For the rarefaction waves $(v^{R_i}, u^{R_i}, \theta^{R_i})(\frac{x}{t})$ determined by \eqref{Equ2} and \eqref{IR}, there holds
\begin{align}
\lim_{t\rightarrow+\infty}\sup_{x\in\mathbb{R}}|(v_{r_i},u_{r_i},\theta_{r_i})(x,t)-(v^{R_i}, u^{R_i}, \theta^{R_i})(\frac{x}{t})|=0.
\end{align}
\end{Lem}

Here we introduce an essential inequality about the heat kernel from \cite{Huang2010}. Define
\begin{align}\label{eq47}
\begin{aligned}
w(x,t)=(1+t)^{-\frac{1}{2}}e^{-\frac{ \sigma x^2}{1+t}}, \ \ g(x,t)=\int_{-\infty}^{x}w(y,t)dy,
\end{aligned}
\end{align}
for $\sigma>0$. It is easy to know that
\begin{align}\label{eq48}
\begin{aligned}
4\sigma g_t=w_x,\ \ \|g(\cdot, t)\|_{L^\infty}=\sqrt{\pi}\sigma^{-\frac{1}{2}}.
\end{aligned}
\end{align}

\begin{Lem}\label{lemm9}(\cite{Huang2010})
Assume that $h(x,t)$ satisfies
\begin{align*}
h_x\in L^2(0, T;L^2(\mathbb{R})), \ \ h_t\in L^2(0,T;H^{-1}(\mathbb{R})),
\end{align*}
for any $0<T\le +\infty$. Then there holds
\begin{align}\label{eq50}
\begin{aligned}
\int_0^T&\int_{\mathbb{R}}h^2w^2dxdt\\
&\le4\pi\|h(0)\|^2+4\pi\sigma^{-1}\int_0^T\|h_x(t)\|^2dt+8\sigma\int_0^T\big<h_t, hg^2\big>_{H^{-1}\times H^1}dt.
\end{aligned}
\end{align}
\end{Lem}

The periodic solutions $(\bar v_\pm, \bar u_\pm, \bar\theta
 _\pm)$ and $(\bar v^o_\pm, \bar u_\pm^o,
 \bar \theta
 ^o_\pm)$ satisfy the following properties:

\begin{Lem}\label{lemm1}(\cite{YuanYuan2021}) 
If periodic functions $\phi_j(x,0)\in H^3((0,1))\ ( j=1,2,3)$ satisfy \eqref{zero}, there exists $\epsilon_0>0$ such that if
\begin{align}\label{eq25}
\epsilon_1:=\|(\phi_{1}, \phi_{2}, \phi_{3})(x,0)\|_{H^3((0,1))}\le \epsilon_0,
\end{align}
then there exists a unique periodic solution
$$(\bar v_\pm, \bar u_\pm, \bar 
\theta
_\pm)(x,t),  \text{and}\ (\bar v^o_\pm, \bar u_\pm^o,
 \bar 
 \theta
 ^o_\pm)(x,t) \in C\big(0,+\infty;H^3((0,1))\big)$$
 to \eqref{Equ1} with the initial data $(\bar v_\pm,\bar u_\pm, \bar E
 _\pm)(x,0)$ and $(\bar v^o_\pm, \bar u_\pm^o, \bar E
 ^o_\pm )(x,0)$, respectively, which satisfy
\begin{align}\label{p2}
\begin{aligned}
 \|(\widetilde v_{\pm}, \widetilde u_\pm, \widetilde E
 _\pm)(\cdot,t)\|_{W^{2,\infty}(\mathbb{R})}\le C\epsilon_1e^{-2\alpha t},\ \ t\ge0, \\
 \text{and}\ \
  \|(\widetilde v^o_{\pm}, \widetilde u^o_\pm, \widetilde E
  ^o_\pm)(\cdot,t)\|_{W^{2,\infty}(\mathbb{R})}\le C\epsilon_1e^{-2\alpha t},\ \ t\ge0, 
 \end{aligned}
 \end{align}
respectively, where constants $\alpha>0$ and $C>0$ are independent of $t$ and $\epsilon_1$.
\end{Lem}

The Lemma \ref{lemm1} is proved by the standard energy method (see \cite{YuanYuan2021}).
Let $$I_1:=|\eta(1-\eta)|,
\ I_2:=|\nabla_{t,x}^m \eta |.$$

By
 direct calculations, 
one can obtain
\begin{equation}\label{equ17}
\begin{cases}
\bar v_{t}-\bar u_{x}=F,\\
\bar u_{t}+\bar p_{x}={\mu}(\frac{\bar u_{x}}{\bar v})_{x}+G,  \qquad  \qquad \qquad \ \ \ \ \ \ \ \ \ \ \ \ x\in\R,\  t>0, \\
\frac{R}{\gamma-1}\bar \theta_t+\bar p\bar u_x={\kappa}(\frac{{\bar\theta}_{x}}{\bar v})_{x}+{\mu}\frac{\bar u^2_{x}}{\bar v}+H,
\end{cases}
\end{equation}
with 
\begin{align}\label{equ180}
\begin{aligned}
	&F=\eta_t(\widetilde v_-^o-\widetilde v_-)-\eta_x(\widetilde u_-^o-\widetilde u_-)
	+\eta_t( \widetilde v_+-\widetilde v_+^o)-\eta_x( \widetilde u_+-\widetilde u_+^o)\\
	&\qquad+\eta_{t}(\widetilde v_+^o-\widetilde v_-^o)-\eta_x(\widetilde u_+^o-\widetilde u_-^o)\\
	&\quad \le C\max\big\{ I_2
	\big\}\epsilon_0e^{-2\alpha t}, 
\end{aligned}
\end{align}
\begin{align}\label{equ181}
\begin{aligned}	
	&G=-\eta_t\widetilde u_- +\eta_{t}\widetilde u_+ +\eta(\bar u_{+})_{t}
	+(1-\eta_{u_{r_1}})(\bar u_{-})_{t}+Q_1\\
	&\qquad +(\bar p-p_{r_1}-p_{r_3}-p^{cd})_x+\mu\big(\frac{(u^{cd})_x}{v^{cd}}-\frac{\bar u_x}{\bar v}\big)_x\\
	&\quad\le \Big\{C\max\{I_1,I_2\}\max\big\{\epsilon_0e^{-2\alpha t}, \delta e^{-c_0(|x|+t)}\big\}+O(\delta)(1+t)^{-\frac{3}{2}}e^{-\frac{C_2x^2}{1+t}}\Big\}\\
	&\qquad+\frac{\mu}{\bar v}(|u_{{r_1}xx}|+|u_{{r_3}xx}|)
	+C\big(|(u^{cd})_{xx}|+|(u^{cd})_x|+|v_{{r_1}x}|+|v_{{r_3}x}|\\&\qquad+|(v^{cd})_{x}|\big)\max\big\{\epsilon_0e^{-2\alpha t}, \delta e^{-c_0(|x|+t)}\big\}\\
	&\quad:=G_1+G_2+G_3, 
\end{aligned}
\end{align}
\begin{align}\label{equ182}
\begin{aligned}
	&H=\frac{R}{\gamma-1}\big[\eta_{t}(\widetilde \theta_-^o-\widetilde \theta_-)+\eta_{t}(\widetilde \theta_+-\widetilde \theta_+^o)+\eta_{ t}(\widetilde \theta_+^o-\widetilde \theta_-^o)\\
	&\qquad+(1-\eta)(\bar\theta_{-})_{t} +\eta(\bar\theta_{+})_{t}\big]
	+
	\big[-\eta_{x}\widetilde u_- +\eta_{x}\widetilde u_+\\
	&\qquad+(1-\eta)(\widetilde u_{-})_{x}+\eta(\widetilde u_{+})_{x}
	+(u^{cd}+u_{r_1}+u_{r_3}-u^o)_x\big]\bar p\\
&\qquad-p_{r_1}u_{{r_1}x}-p_{r_3}u_{{r_3}x}-p^{cd}(u^{cd})_{x}+Q_2 -u^{cd}Q_1\\
	&\qquad+\kappa\big(\frac{(\theta^{cd})_x}{v^{cd}}-\frac{\bar\theta_x}{\bar v}\big)_x+\mu\big( \frac{\big((u^{cd})_x\big)^2}{v^{cd}}-\frac{(\bar u_x)^2}{\bar v}\big)\\
	&\quad \le\Big\{ C\max\{I_1,I_2\}\max\big\{\epsilon_0e^{-2\alpha t}, \delta e^{-c_0(|x|+t)}\big\}+O(\delta)(1+t)^{-2}e^{-\frac{C_2x^2}{1+t}}\Big\}\\
	&\qquad+\frac{\kappa}{\bar v}(|\theta_{{r_1}xx}|+|\theta_{{r_3}xx}|)
	+C\big[|(u^{cd})_{x}|+u_{{r_1}x}+u_{{r_3}x}+|(\theta^{cd})_x|+|\theta_{{r_1}x}|\\
	&\qquad+|\theta_{{r_3}x}|+|(\theta^{cd})_{xx}|+\big((u^{cd})_x\big)^2\big]\max\big\{\epsilon_0e^{-2\alpha t}, \delta e^{-c_0(|x|+t)}\big\}\\
	&\qquad+C\big[(u_{r_1x})^2+(u_{r_3x})^2+\big((u^{cd})_x\big)^2\big]\\
	&\quad:=H_1+H_2+H_3+H_4, 
\end{aligned}
\end{align}

and note that
\begin{align}\label{eq180}
\|F\|_{L^1}\le C\epsilon_0
e^{-2\alpha t},
\end{align}
\begin{align}\label{eq181}
\|G_2\|_{L^1}\le C\delta^{\frac{1}{8}}(1+t)^{-\frac{7}{8}}, \ \
\|G_3\|_{L^1}\le C\epsilon_\delta e^{-\hat Ct},
\end{align}
and
\begin{align}\label{eq182}
\|H_2\|_{L^1}\le C\delta^{\frac{1}{8}}(1+t)^{-\frac{7}{8}}, \ \
\|H_4\|_{L^1}\le C\delta(1+t)^{-1},
\end{align}
 where $\bar p=R\frac{\bar\theta}{\bar v}$,  
 $\hat C=\min\{2\alpha, c_0, C_2\}$ and $\epsilon_\delta$ is suitably small constant depending on $\epsilon_0$ and $\delta$.

\section{Proof of Theorem \ref{thm1}}

In this section, we shall prove Theorem \ref{thm1}. Note that, by Lemma \ref{lemm2} and \eqref{ap}, we can obtain that
\begin{equation}\label{error1}
\lim_{t\rightarrow\infty}\sup_{x\in\mathbb{R}}|(\widetilde V,\widetilde U,\widetilde \Theta)(x,t)-(\bar V,\bar U, \bar \Theta)(x,t)|=0.
\end{equation}
Additionally, by Lemma \ref{lemm1}, we can easily verify that
\begin{equation}\label{error2}
\lim_{t\rightarrow\infty}\sup_{x\in\mathbb{R}}|(\widetilde V,\widetilde U,\widetilde \Theta)(x,t)-(\bar v,\bar u, \bar \theta)(x,t)|=0.
\end{equation}
Hence, to show \eqref{eq030} in Theorem \ref{thm1}, it suffices to prove that
\begin{equation}\label{error3}
\lim_{t\rightarrow\infty}\sup_{x\in\mathbb{R}}|(v,u,\theta)(x,t)-(\bar v,\bar u, \bar \theta)(x,t)|=0.
\end{equation}

Therefore, we will focus on the reformulated system described by the perturbation $(\phi, \psi, \zeta)=(v-\bar v, u-\bar u, \theta-\bar\theta)(x,t)$:
\begin{equation}\label{equ20}
\begin{cases}
\phi_{t}-\psi_{x}=-F,\\
\psi_{t}+(p-\bar p)_{x}={\mu}(\frac{ u_{x}}{ v}-\frac{\bar u_{x}}{\bar v})_{x}-G,  \qquad  \qquad \qquad \ \ \ \ \ \ \ \ \ \ \ \ x\in\R,\  t>0, \\
\frac{R}{\gamma-1} \zeta_t+pu_x-\bar p\bar u_x={\kappa}(\frac{{\theta}_{x}}{v}-\frac{{\bar\theta}_{x}}{\bar v})_{x}+{\mu}(\frac{ u^2_{x}}{ v}-\frac{\bar u^2_{x}}{\bar v})-H,
\end{cases}
\end{equation}
with
\begin{align}\label{equa20}
\begin{aligned}
(\phi_0,\psi_0,\zeta_0)(x)=\Big(0,0,&-\frac{\gamma-1}{R}\phi_2(x)
u^{cd}(x,0)\Big)\in H^2(\mathbb{R}).
\end{aligned}
\end{align}


We first state the local existence of solution $(\phi,\psi,\zeta)$ in the solution space $X_{\underline\iota,\iota}(0,+\infty)=\cup_{T>0}X_{\underline\iota,\iota}(0,T)$, where
$$X_{\underline\iota,\iota}(0,T)=\big\{(\phi,\psi,\zeta)\in X(I), \|(\phi,\psi,\zeta)\|_1\le \iota, \inf\limits_{x,t}(v+\phi_0)\ge \underline\iota>0\big\},$$
for positive constants $\iota$ and $\underline\iota.$

\begin{Prop}\label{pro151}{(Local existence)}
There exists a positive constant $b$ 
and $T^*=T^*(\underline\iota,\iota)
$ such that if $\|(\phi_0,\psi_0,\zeta_0)\|_1\le \iota
$ and $\inf_{\mathbb{R}.}(\bar v+\phi_0)\ge \underline\iota$, then 
there exists a unique solution 
$(\phi, \psi, \zeta)\in X_{\frac{1}{2}\underline\iota,b\iota}
(0, T^*)$ to \eqref{equ20}.
\end{Prop}

Since the proof of Proposition \ref{pro151} is similar to that in \cite{Huang2004}, we omit it for simplicity. Assume that 
\begin{align}\label{eq32}
 \underline\iota^o=\frac{1}{2}\min\{v_-,v_+\}, \ \ N(T)=\sup\limits_{t\in[0,T]}\|(\phi, \psi, \zeta)(\cdot, t)\|^2_{1}<\chi^2, \ \ 
\forall T>0,
\end{align}
where $\chi$ is a small constant.
To prove the global existence, we only need to 
 establish the following proposition:
\begin{Prop}\label{pro15}{(A priori estimate)}
There exist suitably small positive constants $\epsilon_0$, $\chi$ and $\delta_0\le\min\big\{\bar\delta,\delta_1,1\big\}$, such that for any $T>0$ and $(\phi, \psi,\zeta)\in X_{\frac{1}{2}\underline\iota^o,\chi}(I) 
$ 
 satisfying \eqref{eq32} and $|\theta_--\theta_+|=\delta\le\delta_0$,
there holds
\begin{align}\label{eq91}
\begin{aligned}
\sup_{t\in[0,T]}&\|(\phi,\psi,\zeta)(t)\|_1^2+\int_0^T\big(\|\phi_x(s)\|^2+\|(\psi_x, \zeta_x)(s)\|^2_1\big)ds\\
&+\int_0^T\int_{\mathbb{R}}(\phi^2+\zeta^2)(u_{r_1x}+u_{r_3x})dxds\le C\big(\chi\epsilon_\delta+\|(\phi_0, \psi_0,\zeta_0)\|_1^2\big),\end{aligned}
\end{align}
for some positive constant $C$, where $\epsilon_\delta$ is suitably small and only depends on $\epsilon_0$ and $\delta$.
\end{Prop}

Proposition \ref{pro15} is a direct consequence of the following lemmas. The following key Lemma is established by basic energy estimates.
\begin{Lem}\label{lemm15}
 Assume $(\phi,\psi,\zeta)\in X_{\frac{1}{2}\underline\iota^o,\chi}(I)$ satisfying 
\eqref{eq32} with suitably small $\chi$ and $\delta_0$, 
then for any $T>0$, $ \forall t\in[0,T] $, it holds
\begin{align}\label{eq209}
\begin{aligned}
&\|(\phi,\psi,\zeta)(t)\|^2+\int_0^t\|(\psi_x, \zeta_x)(s)\|^2ds
+\int_0^t\int_{\mathbb{R}}(u_{{r_1}x}+u_{{r_3}x})(\phi^2+\zeta^2)dxds\\
&\le C\|(\phi_0,\psi_0,\zeta_0)\|^2+(C
+C_{\beta_1})\chi\epsilon_\delta
+C\beta_1
\int_0^t\|\phi_x(s)\|^2ds,
\end{aligned}
\end{align}
where $\beta_1$ 
 is a small constant to be determined  later and $C_{\beta_1}$ 
 is a positive constant depending on $\beta_1$ . 
\end{Lem}

Lemma \ref{lemm15} was derived directly from the following Lemmas by choosing $\epsilon_\delta$ 
suitably small.

\begin{Lem}\label{lemm5}
For $T>0$ and $(\phi,\psi,\zeta)\in X_{\frac{1}{2}\underline\iota^o,\chi}(I)$ satisfying 
 \eqref{eq32} with suitably small $\chi$, 
 then 
  for any $t\in[0,T]$, we have
\begin{align}\label{eq29}
\begin{aligned}
&\|(\phi,\psi,\zeta)(t)\|^2+\int_0^t\|(\psi_x, \zeta_x)(s)\|^2ds
+\int_0^t\int_{\mathbb{R}}(u_{{r_1}x}+u_{{r_3}x})(\phi^2+\zeta^2)dxds\\
&\le C\|(\phi_0,\psi_0,\zeta_0)\|^2+(C
+C_{\beta_1})\chi\epsilon_\delta
+C\beta_1
\int_0^t\|\phi_x(s)\|^2ds\\
&\qquad+
C\int_0^t\int_{\mathbb{R}}\delta(1+s)^{-1}e^{-\frac{C_2x^2}{1+s}}(\phi^2+\zeta^2)dxds,
\end{aligned}
\end{align}
where $\beta_1$ 
 is a small constant to be determined  later and $C_{\beta_1}$ 
 is a positive constant depending on $\beta_1$. 
\end{Lem}

\begin{proof}
Let $$\Phi(Z)=Z-1-\ln Z,\ \ \Phi(1)=\Phi'(1)=0.$$
Taking $\eqref{equ20}_2\times\psi$ and $\eqref{equ20}_3\times(1-\frac{\bar\theta}{\theta})$, we have
\begin{align}\label{eq30}
\begin{aligned}
&\big\{\frac{\psi^2}{2}+\frac{R}{\gamma-1}\bar\theta\Phi(\frac{\theta}{\bar\theta})+R\bar\theta\Phi(\frac{v}{\bar v})\big\}_t+\bar p(u_{{r_1}x}+u_{{r_3}x})
\big(\Phi(\frac{\theta\bar v}{v\bar\theta})+\gamma\Phi(\frac{v}{\bar v})\big) \\
&\qquad \ \ \ +\mu\frac{\psi_x^2}{v}+\kappa\frac{\zeta_x^2}{v\theta}\\
&=\Big\{(\gamma-1)\big[\kappa(\frac{\bar\theta_x}{\bar v})_x+\mu\frac{(\bar u_x)^2}{\bar v}+H\big]\big[\Phi(\frac{v}{\bar v})-\frac{1}{\gamma-1}\Phi(\frac{\bar\theta}{\theta})\big]-\bar p(\triangle_1\\
&\qquad +u^{cd}_{x})
\big(\Phi(\frac{\theta\bar v}{v\bar\theta})+\gamma\Phi(\frac{v}{\bar v})\big)\Big\}-\big[\mu\psi_x\bar u_x(\frac{1}{v}-\frac{1}{\bar v})-\mu\frac{\zeta}{\theta}(\frac{u_x^2}{v}-\frac{\bar u_x^2}{\bar v})\big]\\
&\qquad-\big[\kappa\frac{\zeta_x\bar\theta_x}{\theta}(\frac{1}{v}-\frac{1}{\bar v})-\kappa\frac{\zeta\theta_x}{\theta^2}(\frac{\theta_x}{v}-\frac{\bar\theta_x}{\bar v})\big]-\frac{\bar p\phi^2}{v\bar v}F+R\bar\theta F(\frac{1}{v}-\frac{1}{\bar v})-G\psi\\
&\qquad -\frac{\zeta}{\theta}H+\big\{\mu\psi(\frac{u_x}{v}-\frac{\bar u_x}{\bar v})-(p-\bar p)\psi+\kappa\frac{\zeta_x}{v}+(\frac{1}{v}-\frac{1}{\bar v})\bar\theta_x\big\}_x\\
&=J_{11}+J_{12}+J_{13}-\frac{\bar p\phi^2}{v\bar v}F+R\bar\theta F(\frac{1}{v}-\frac{1}{\bar v})-G\psi-\frac{\zeta}{\theta}H+J_{14},
\end{aligned}
\end{align}
where
\begin{align}\label{eq31}
\begin{aligned}
\triangle_1=\eta_x(\widetilde u_+ -\widetilde u_-)+(1-\eta)\widetilde u_{-x}
+\eta \widetilde u_{+x}\le C\epsilon_0e^{-2\alpha t}.
\end{aligned}
\end{align}
It is easy to check that
\begin{align}\label{eq34}
\begin{aligned}
J_{11}\le C&(\triangle_1+u^{cd}_{x})
(\phi^2+\zeta^2)+C\big[
\omega_{r_1x}^2+
\omega_{r_3x}^2+(\omega^{cd}_{x})^2+\theta^{cd}_{xx}+\theta_{{r_1}xx}\\
&+\theta_{{r_3}xx}\big]
(\phi^2+\zeta^2)
+C|H|(\phi^2+\zeta^2),
\end{aligned}
\end{align}
\begin{align}\label{eq33}
\begin{aligned}
J_{12}&\le\mu\frac{\psi_x^2}{4v}+\mu\frac{C}{v}(\bar u_x\frac{\phi}{\bar v})^2+\mu\frac{\zeta}{\theta}(\bar u_x^2\frac{-\phi}{v\bar v}+\frac{2\bar u_x\psi_x}{v}+\frac{\psi_x^2}{v})\\
&\le\mu\frac{\psi_x^2}{2v}+C\big[\epsilon_0e^{-2\alpha t}
+(u_{{r_1}x})^2+(u_{{r_3}x})^2+(u^{cd}_{x})^2\big](\phi^2+\zeta^2)+C|\zeta|\psi_x^2,
\end{aligned}
\end{align}
\begin{align}\label{eq35}
\begin{aligned}
J_{13}&\le\kappa\frac{\zeta_x\bar\theta_x\phi}{\theta v\bar v}+\kappa(-\frac{\phi}{v\bar v})\frac{\zeta\bar\theta_x^2}{\theta^2}+2\kappa\frac{\zeta_x\zeta\bar\theta_x}{\theta v}+\kappa\frac{\zeta\zeta_x^2}{\theta^2 v}-\kappa\frac{\zeta_x\zeta\bar\theta_x}{\theta^2\bar v}-\frac{\zeta\bar\theta_x^2}{\bar v\theta^2}+\frac{\zeta\bar\theta_x^2}{v\theta^2}\\
&\le\frac{1}{4}\kappa\frac{\zeta^2_x}{\theta v}+C\bar\theta_x^2(\phi^2+\zeta^2)+C|\zeta|\zeta_x^2+\beta\zeta_x^2+C_\beta\bar\theta_x^2\zeta^2\\
&\le\frac{1}{4}\kappa\frac{\zeta^2_x}{\theta v}+(C+C_\beta)[\epsilon_0e^{-2\alpha t}
+\theta_{{r_1}x}^2+\theta_{{r_3}x}^2+(\theta^{cd}_x)^2](\phi^2+\zeta^2)+C|\zeta|\zeta_x^2+\beta\zeta_x^2,
\end{aligned}
\end{align}
where $\beta$ 
 is a small constant to be determined  later, $C_{\beta}:=C(\beta)$ is a positive constant depending on $\beta$,  
and $\omega_{r_1x}=(v_{{r_1}x}, u_{{r_1}x}, \theta_{{r_1}x}),\ \omega_{r_3x}=(v_{{r_3}x}, u_{{r_3}x}, \theta_{{r_3}x}),\ \omega^{cd}_{x}=(v^{cd}_{x}, u^{cd}_{x}, \theta^{cd}_{x}).$ 
Using
\begin{align}\label{eq43}
\|(\phi, \psi, \zeta)\|_{L^\infty}\le \sqrt{2}\|(\phi, \psi, \zeta)\|^{\frac{1}{2}}\|(\phi_x, \psi_x, \zeta_x)\|^{\frac{1}{2}}, 
\end{align}
and integrating \cref{eq30} over $[0, t] \times \mathbb{R}$, we have
\begin{align}\label{eq37}
\begin{aligned}
&\|(\phi,\psi,\zeta)(\cdot, t)\|^2+\int_0^t\int_{\mathbb{R}}\mu\frac{\psi_x^2}{v}+\kappa\frac{\zeta^2_x}{\theta v}dxds+\int_0^t\int_{\mathbb{R}}(u_{{r_1}x}+u_{{r_3}x})(\phi^2+\zeta^2)dxds\\
&\le\|(\phi_0,\psi_0,\zeta_0)\|^2+\beta\int_0^t\int_{\mathbb{R}}\zeta^2_xdxds+C\int_0^t\int_{\mathbb{R}}|\zeta|(\psi_x^2+\zeta_x^2)dxds\\
&-\int_0^t\int_{\mathbb{R}}\frac{\bar p\phi^2}{v\bar v}F+R\bar\theta F\frac{\phi}{v\bar v}+G\psi +\frac{\zeta}{\theta}Hdxds+(C+C_\beta)\int_0^t\int_{\mathbb{R}}\big[\epsilon_0e^{-2\alpha t}+\omega_{r_1x}^2\\
&+\omega_{r_3x}^2+|u_x^{cd}|^2
+|\theta_{{r_1}xx} |+|\theta_{{r_3}xx}|\big]
(\phi^2+\zeta^2)dxds+(C+C_\beta)\int_0^t\int_{\mathbb{R}}(|u_x^{cd}|\\
&+|\theta_x^{cd}|^2+|v_x^{cd}|^2+|\theta^{cd}_{xx}|)
(\phi^2+\zeta^2)dxds\\
&\le\|(\phi_0,\psi_0,\zeta_0)\|^2+(\beta+C\chi
)\int_0^t\int_{\mathbb{R}}(\psi_x^2+\zeta_x^2)dxds+\sum_{i=1}^3\tilde J_i.
\end{aligned}
\end{align}
Then by \eqref{eq180}-\eqref{eq182}, \eqref{eq32} and \eqref{eq43}, we obtain that
\begin{align}\label{eq45}
\begin{aligned}
|\tilde J_1|&=|\int_0^t\int_{\mathbb{R}}\frac{\bar p\phi^2}{v\bar v}F+R\bar\theta F\frac{\phi}{v\bar v}+G\psi +\frac{\zeta}{\theta}Hdxds|\\
&\le C\int_0^t\epsilon_\delta
\big[\|\phi\|_{L^\infty}^2+\|(\phi,\psi,\zeta)(s)\|_{L^\infty}
\big]e^{-\hat Cs}
+\|\psi\|^{\frac{1}{2}}\|\psi_x\|^{\frac{1}{2}}\|G_2\|_{L^1}\\
&\quad+\|\zeta\|^{\frac{1}{2}}\|\zeta_x\|^{\frac{1}{2}}\|(H_2,H_4)\|_{L^1}ds\\
&\le (C+C_\beta+C_{\beta_1}
)\chi\epsilon_\delta+\int_0^t\beta_1
\|\phi_x\|^2+\beta\|(\psi_x,\zeta_x)\|^2ds\\
&\quad+C_\beta\int_0^t\|\psi\|^{\frac{2}{3}}\|G_2\|_{L^1}^{\frac{4}{3}}+\|\zeta\|^{\frac{2}{3}}\|(H_2,H_4)\|_{L^1}^{\frac{4}{3}}ds\\
&\le (C+C_\beta+C_{\beta_1})
\chi\epsilon_\delta+\int_0^t\beta_1
\|\phi_x\|^2+\beta\|(\psi_x, \zeta_x)\|^2ds,
\end{aligned}
\end{align}
where $\beta_1$ 
 is a small constant 
  and $C_{\beta_1}:=C(\beta_1)$ is a positive constant depending on $\beta_1$. 
From the Lemma \ref{lemm2} and \eqref{equ6}, 
we get
\begin{align}\label{eq40}
\begin{aligned}
&\|(u_x^{cd})^2\|_{L^1}\le C\delta(1+t)^{-\frac{3}{2}}, 
\\ &
\|(\omega_{r_1x}^2, \omega_{r_3x}^2,u_{{r_1}xx}, u_{{r_3}xx}, \theta_{{r_1}xx}, \theta_{{r_3}xx})\|_{L^1}\le C\delta^{\frac{1}{8}}(1+t)^{-\frac{7}{8}},
\end{aligned}
\end{align}
then
\begin{align}\label{equ40}
\begin{aligned}
|\tilde J_2|&\le \int_0^t \beta_1\|\phi_x\|^2+ \beta\|\zeta_x\|^2ds+(C_\beta+C_{\beta_1})\delta^{\frac{1}{4}}\|(\phi,\zeta)\|^2\int_0^t
(1+t)^{-\frac{7}{4}}ds\\
&\le
\int_0^t \beta_1\|\phi_x\|^2+ \beta\|\zeta_x\|^2
ds+(C_\beta+C_{\beta_1})\chi\delta^{\frac{1}{4}}.
\end{aligned}
\end{align}
From \eqref{equ6}, we obtain
\begin{align}\label{eq46}
\begin{aligned}
|\tilde J_3|&=(C+C_\beta)\int_0^t\int_{\mathbb{R}}(|u_x^{cd}|+|\theta_x^{cd}|^2+|v_x^{cd}|^2+|\theta^{cd}_{xx}|)
(\phi^2+\zeta^2)dxds\\
&\le (C+C_\beta)\int_0^t\int_{\mathbb{R}}\delta(1+s)^{-1}e^{-\frac{C_2x^2}{1+s}}(\phi^2+\zeta^2)dxds.
\end{aligned}
\end{align}
Choosing $\beta$ suitably small, we get \eqref{eq29}
 from \eqref{eq37}-\eqref{eq46}.
\end{proof}
In order to finish the proof of Lemma \ref{lemm15}, we have to estimate the last term in \eqref{eq29} based on the Lemma \ref{lemm9} as follows.

\begin{Lem}\label{lemm11}
Assume that positive constant $\sigma\in(0, \frac{C_2}{4}]$ and $w$ defined in \eqref{eq47}, then there holds
\begin{align}\label{eq54}
\begin{aligned}
&\int_0^t\int_{\mathbb{R}}\big(\phi^2+\psi^2+\zeta^2\big)w^2dxds\\
&\le C+C\int_0^t\big(\|\phi_x\|^2+\|\psi_x\|^2+\|\zeta_x\|\big)ds+C\int_0^t\int_{\mathbb{R}}(\phi^2+\zeta^2)(u_{r_1x}+u_{r_3x})dxds,
\end{aligned}
\end{align}
where $C$ is some positive constant depending on $\sigma$, $\chi$ and $\epsilon_\delta$.
\end{Lem}
The proof of this Lemma \ref{lemm11} is similar to the Lemma 5 in \cite{Huang2010}, so we omit it for simplicity.
Multiplying \eqref{eq54} by 
$C\delta$ ($<C\delta_0$), 
we quickly obtain the Lemma \ref{lemm15}.

\begin{Lem}\label{lemm6}
Assume that $(\phi,\psi,\zeta)\in X_{\frac{1}{2}\underline\iota^o,\chi}(I)$ satisfying 
\eqref{eq32} with suitably small $\chi$, then for any $T>0$, there holds
\begin{align}\label{eq38}
\begin{aligned}
\|\phi_x(t)\|^2+\int_0^t\|\phi_x(s)\|^2ds\le C\big(\chi\epsilon_\delta
+\|\psi_0,\zeta_0\|^2+\|\phi_0\|_1^2\big),\ \  \forall t\in[0,T].
\end{aligned}
\end{align}
\end{Lem}
\begin{proof} Let $\hat v=\frac{v}{\bar v}$, then
\begin{align}\label{eq73}
\begin{aligned}
\frac{\hat v_x}{\hat v}=\frac{\phi_x}{v}&-\frac{\bar v_x\phi}{v\bar v},\ \
\big(\frac{\hat v_t}{\hat v}\big)_x=\big(\frac{\hat v_x}{\hat v}\big)_t, \ \ \big(\frac{u_x}{v}-\frac{\bar u_x}{\bar v}\big)_x=\big(\frac{\hat v_t}{\hat v}+\frac{F}{\bar v}\big)_x,\\ &\big(p-\bar p\big)_x=R\frac{\zeta_x}{v}-R\frac{\bar\theta_x\phi}{v\bar v}-(p-\bar p)\frac{\bar v_x}{\bar v}-p\frac{\hat v_x}{\hat v}.
\end{aligned}
\end{align}
We rewrite \eqref{equ20}$_2$ as
\begin{align}\label{eq74}
\begin{aligned}
\mu(\frac{\hat v_x}{\hat v})_t+\mu(\frac{F}{\bar v})_x-G-\psi_t-R\frac{\zeta_x}{v}+R\frac{\bar\theta_x\phi}{v\bar v}+(p-\bar p)\frac{\bar v_x}{\bar v}+p\frac{\hat v_x}{\hat v}=0.
\end{aligned}
\end{align}
Multiplying \eqref{eq74} by $\frac{\hat v_x}{\hat v}$, and integrating the resulting equation over $(0,t)\times\mathbb{R}$, we have
\begin{align}\label{eq75}
\begin{aligned}
&\int_{\mathbb{R}}\big(\frac{\mu}{2}(\frac{\hat v_x}{\hat v})^2-\psi\frac{\hat v_x}{\hat v}\big)(t)dx+\int_0^t\int_{\mathbb{R}}\frac{R\theta}{v}(\frac{\hat v_x}{\hat v})^2dxds\\
&=\int_{\mathbb{R}}\big(\frac{\mu}{2}(\frac{\hat v_x}{\hat v})^2-\psi\frac{\hat v_x}{\hat v}\big)(0)dx+\int_0^t\int_{\mathbb{R}}\psi_x\big(\frac{\psi_x}{v}-\frac{\bar u_x\phi}{v\bar v}\big)-\psi_x\frac{F}{\bar v}\\
&\quad+\big[G-\mu\big(\frac{F}{\bar v}\big)_x+R\frac{\zeta_x}{v}-R\frac{\bar\theta_x\phi}{v\bar v}-(p-\bar p)\frac{\bar v_x}{\bar v}\big]\big(\frac{\phi_x}{v}-\frac{\bar v_x\phi}{v\bar v}\big)
dxds\\
&\le C(\|\phi_0\|^2_1+\|\psi_0\|^2)-\int_0^t\int_{\mathbb{R}}\frac{G_x\phi}{v}dxds+C\int_0^t\Big\{\|\psi_x\|^2+(1+C_\beta)\|\zeta_x\|^2\\
&\quad+\int_{\mathbb{R}}(1+C_\beta)(\bar v_x^2+\bar\theta_x^2)\|\phi\|_{L^{\infty}}^2+\|\phi\|_{L^\infty}^2(C_\beta G^2+\bar v_x^2)+\bar u_x^2\phi^2
dx\\
&\quad+(\chi+\beta)\|\phi_x\|^2+\int_{\mathbb{R}}(1+\bar v_x^2)F^2+\big(\epsilon_0e^{-2\alpha s}+\delta(1+s)^{-\frac{1}{2}}e^{-\frac{C_2x^2}{1+s}}\big)|\bar v_x\phi|\\
&\quad+F_x^2+ (1+C_\beta)\bar v_x^2\zeta^2dx\Big\}ds\\
&\le C(\|\phi_0\|^2_1+\|\psi_0\|^2)+(C+C_\beta)\int_0^t(\|\psi_x\|^2+\|\zeta_x\|^2)ds+(\chi+C
\beta)
\int_0^t\|\phi_x\|^2ds\\
&\quad+(C+C_\beta)\chi\epsilon_\delta+C(1+C_\beta)\delta\int_0^t\int_\mathbb{R}(1+s)^{-1}e^{-\frac{C_2x^2}{1+s}}\phi^2dxds\\
&\quad +\int_0^t\int_{\mathbb{R}}\big|\frac{G_x\phi}{v}\big|dxds,
\end{aligned}
\end{align}
where $\|( v_{r_ix}^2,u_{r_ix}^2, \theta_{r_ix}^2)\|_{L^1}\le C\delta^{\frac{1}{8}}(1+t)^{-\frac{7}{8}}$ is used in the last inequality.
From \eqref{equ181}, we get
\begin{align}\label{eq76}
\begin{aligned}
G_x&\le \Big\{C\max\{I_1,I_2\}\max\big\{\epsilon_0e^{-2\alpha t}, \delta e^{-c_0(|x|+t)}\big\}+O(\delta)(1+t)^{-2
	}e^{-\frac{C_2x^2}{1+t}}\Big\}\\
	&\quad+\frac{\mu}{\bar v}(|u_{{r_1}xxx}|+|u_{{r_3}xxx}|)
	+C\big(|u_{xx}^{cd}|+|u^{cd}_{xxx}|+|u^{cd}_x|+|v_{xx}^{cd}|+|v_{{r_1}x}|\\
	&\quad+|v_{{r_3}x}|+|v_{{r_1}{xx}}|+|v_{{r_3}{xx}}|+|v_{x}^{cd}|\big)\max\big\{\epsilon_0e^{-2\alpha t}, \delta e^{-c_0(|x|+t)}\big\}\\
	&\quad:=G^1_x+G^2_x+G^3_x,
\end{aligned}
\end{align}
and
\begin{align}\label{equ76}
\begin{aligned}
&\|G_x^1\|\le C\epsilon_\delta e^{-\hat Ct}+C\delta(1+t)^{-\frac{7}{4}},\\
&\|G_x^2\|_{L^1}\le 
C\delta^{\frac{1}{8}}(1+t)^{-\frac{7}{8}},\ \ \|G_x^3\|\le C\epsilon_\delta e^{-\hat Ct}.
\end{aligned}
\end{align}
And there exists
\begin{align}\label{eq77}
\begin{aligned}
\int_0^t\int_{\mathbb{R}}\big|\frac{G_x\phi}{v}\big|dxds
&\le C\int_0^t\Big\{\|(G^1_x,G^3_x)(s)\|\|\phi\|+\|\phi\|^{\frac{1}{2}}\|\phi_x\|^{\frac{1}{2}}\|G^2_x\|_{L^1}\Big\}ds\\
&\le C\chi\epsilon_\delta+\int_0^t\Big\{\beta\|\phi_x\|^2+C_\beta\|\phi\|^{\frac{2}{3}}\|G^2_x\|_{L^1}^{\frac{4}{3}}\Big\}ds\\
&\le\beta\int_0^t\|\phi_x\|^2ds+(C_\beta+C)\chi\epsilon_\delta.
\end{aligned}
\end{align}

Then combining \eqref{eq73}, \eqref{eq75}, \eqref{eq77} with 
 \eqref{eq29}, 
 we deduce \eqref{eq38}
 by choosing $\beta$, $\delta_0$ and $\epsilon_0$ suitably small.
 \end{proof}

 Now 
 we give the a priori estimate about the higher derivatives of $\psi$ and $\zeta$.

\begin{Lem}\label{lemm7}
Assume that $(\phi,\psi,\zeta)$ satisfying 
 \eqref{eq32} with suitably small $\chi$, then for $t\in[0,T]$, we get
\begin{align}\label{eq39}
\begin{aligned}
\|(\psi_x,\zeta_x)(t)\|^2+\int_0^t\|(\psi_{xx},\zeta_{xx})(s)\|^2ds\le C\chi\epsilon_\delta+C\|(\phi_0, \psi_0,\zeta_0)\|_1^2.
\end{aligned}
\end{align}
\end{Lem}
\begin{proof}
Taking \eqref{equ20}$_2\times(-\psi_{xx})+$\eqref{equ20}$_3\times(-\zeta_{xx})$, we have
\begin{align}\label{eq80}
\begin{aligned}
\frac{1}{2}&\big\{\psi_x^2+\frac{R}{\gamma-1}\zeta_x^2\big\}_t-\big\{\psi_t\psi_x+\frac{R}{\gamma-1}\zeta_t\zeta_x\big\}_x+\mu\frac{\psi_{xx}^2}{v}+\kappa\frac{\zeta_{xx}^2}{v}\\
&=\Big\{(p-\bar p)_x+\mu\frac{v_x\psi_x}{v^2}-\mu\big[\bar u_x(\frac{1}{v}-\frac{1}{\bar v})\big]_x+G\Big\}\psi_{xx}+\Big\{(pu_x-\bar p\bar u_x)\\
&\qquad+\kappa\frac{v_x\zeta_x}{v^2}-\kappa\big[\bar\theta_x(\frac{1}{v}-\frac{1}{\bar v})\big]_x-\mu(\frac{u_x^2}{v}-\frac{\bar u_x^2}{\bar v})+H\Big\}\zeta_{xx}.
\end{aligned}
\end{align}
Integrating \eqref{eq80} over $\mathbb{R}\times(0, t)$, we get
\begin{align}\label{eq81}
\begin{aligned}
&\frac{1}{2}\int_{\mathbb{R}}(\psi_x^2+\frac{R}{\gamma-1}\zeta_x^2)(t)dx+\int_0^t\int_{\mathbb{R}}\mu\frac{\psi_{xx}^2}{v}+\kappa\frac{\zeta_{xx}^2}{v}dxds\\
&=\frac{1}{2} \int_{\mathbb{R}}(\psi_x^2+\zeta_x^2)(0)dx+\int_0^t\int_{\mathbb{R}}\Big\{\big[(p-\bar p)_x+\mu\frac{v_x\psi_x}{v^2}+\mu\big(\frac{\bar u_x\phi}{v\bar v}
\big)_x \big]\psi_{xx}\\
&\quad +\big[(pu_x-\bar p\bar u_x)+
\kappa\frac{v_x\zeta_x}{v^2}+
\kappa\big(\frac{\bar \theta_x\phi}{v\bar v}\big)_x
+\mu(\frac{\phi\bar u_x^2}{v}-\frac{2\bar u_x^2\psi_x+\psi_x^2}{\bar v})\big]\zeta_{xx}\\
&\quad
-(G_x\psi_{x}+H_x\zeta_{x})\Big\}dxds\\
&:=\frac{1}{2}\int_{\mathbb{R}}(\psi_x^2+\zeta_x^2)(0)dx+\sum_{i=1}^{i=3}\bar J_i.
\end{aligned}
\end{align}
From the Lemma \ref{lemm15}, Lemma \ref{lemm11} and Lemma \ref{lemm6}, one can obtain
\begin{align}\label{eq82}
\begin{aligned}
\bar J_1&\le \int_0^t\Big\{C_\beta(\|\phi_x\|^2+\|\zeta_x\|^2)+\beta\|\psi_{xx}\|^2+C_\beta\int_{\mathbb{R}}(\bar v_x^2+\bar\theta_x^2)(\phi^2+\zeta^2)dx\\
&\quad+C_\beta\|\zeta\|^2_{L^\infty}\int_{\mathbb{R}}\phi_x^2dx+C_\beta\int_{\mathbb{R}}(\bar u_{xx}^2+\bar u_x^2+\bar u_x^2\bar v_x^2)\phi^2dx+\|\bar v_x\|_{L^\infty}\|\psi_x\|\|\psi_{xx}\|\\
&\quad+\|\bar u_x\|_{L^{\infty}}\|\phi_x\|\|\psi_{xx}\|+\|\psi_x\|_{L^{\infty}}\|\phi_x\|\|\psi_{xx}\|\Big\}ds\\
&\le C_\beta\int_0^t\|(\phi_x,\zeta_x)\|^2ds+\beta\int_0^t\|\psi_{xx}\|^2ds+C_\beta\int_0^t(\epsilon_\delta
e^{-\hat Cs}+\delta)\|(\phi_x,\psi_x)\|^2ds\\
&\quad+C_\beta\sup_{t\in[0,T]}\|\phi_x(t)\|^4\int_0^t\|\psi_x\|^2ds+C_\beta\chi\epsilon_\delta\\
&\quad+C_\beta\delta\int_0^t\int_{\mathbb{R}}(1+s)^{-1}e^{-\frac{C_2x^2}{1+s}}(\phi^2+\zeta^2)dxds\\
&\le C_\beta\chi\epsilon_\delta+C_\beta
\big(\|(\phi_0,\psi_0, \zeta_0)\|^2+\|\phi_0\|_1^2\big)+\beta\int_0^t\|\psi_{xx}\|^2ds,
\end{aligned}
\end{align}
\begin{align}\label{eq83}
\begin{aligned}
\bar J_2&\le \int_0^t\Big\{\beta\|\zeta_{xx}\|^2+C_\beta\|\psi_x\|^2+C_\beta\int_{\mathbb{R}}\bar u_x^2(\phi^2+\zeta^2+\psi_x^2)+(\bar\theta_{xx}^2+\bar v_x^2\bar\theta_x^2)\phi^2dx\\
&\quad+C\|\bar v_x\|_{L^\infty}\|\zeta_x\|\|\zeta_{xx}\|+C\big(\|\bar\theta_x\|_{L^\infty}+\|\zeta_x\|_{L^{\infty}}+\|\bar\theta_x\|_{L^\infty}\|\phi\|_{L^{\infty}}\big)\|\phi_x\|\|\zeta_{xx}\|\\
&\quad +C\|\psi_x\|_{L^\infty}\|\psi_x\|\|\zeta_{xx}\|\Big\}ds\\
&\le \int_0^t\Big\{\beta\|\zeta_{xx}\|^2+C_\beta\big[(1+\epsilon_\delta)\|\psi_x\|^2+(\chi+\delta)(\|\phi_x\|^2+\|\zeta_x\|^2)\big]\Big\}ds\\
&\quad+C_\beta\chi\epsilon_\delta+C_\beta\sup_{t\in[0,T]}\|\phi_x(t)\|^4\int_0^t\|\zeta_x\|^2ds\\
&\le  C_\beta\chi\epsilon_\delta+C_\beta\big(\|(\phi_0,\psi_0, \zeta_0)\|^2+\|\phi_0\|_1^2\big)+\beta\int_0^t\|\zeta_{xx}\|^2ds,
\end{aligned}
\end{align}
From \eqref{equ182}, we get
\begin{align}\label{eq84}
\begin{aligned}
H_x&\le \Big\{C\max\{I_1,I_2\}\max\big\{\epsilon_0e^{-2\alpha t}, \delta e^{-c_0(|x|+t)}\big\}+O(\delta)(1+t)^{-\frac{5}{2}}e^{-\frac{C_2x^2}{1+t}}\Big\}\\
	&\quad+\frac{\kappa}{\bar v}(|\theta_{{r_1}xxx}|+|\theta_{{r_3}xxx}|)
	+C\big[|u_{xx}^{cd}|+|u_{{r_1}xx}|+|u_{{r_3}xx}|+|\theta_{xx}^{cd}|+|\theta_{{r_1}xx}|\\
	&\quad+|\theta_{{r_3}xx}|+|\theta_{xxx}^{cd}|+|u_{x}^{cd}|+u_{{r_1}x}+u_{{r_3}x}+|\theta_x^{cd}|+|\theta_{{r_1}x}|+|\theta_{{r_3}x}|\\
	&\quad+|\theta_{xx}^{cd}|
	\big]\max\big\{\epsilon_0e^{-2\alpha t}, \delta e^{-c_0(|x|+t)}\big\}\\
	&\quad:=H_x^1+H_x^2+H_x^3,
\end{aligned}
\end{align}
and
\begin{align}\label{equ84}
\|H_x^2\|_{L^1}\le 
C\delta^{\frac{1}{8}}(1+t)^{-\frac{7}{8}},\ \ \|H_x^3\|\le C\epsilon_\delta e^{-\hat Ct}.
\end{align}
Then
\begin{align}\label{eq85}
\begin{aligned}
\bar J_3
&\le 
C\int_0^t\Big\{\|(G^1_x,G^3_x)(s)\|\|\psi_x\|+\|\psi_x\|^{\frac{1}{2}}\|\psi_{xx}\|^{\frac{1}{2}}\|G^2_x\|_{L^1}+\|(H^1_x,H^3_x)(s)\|\|\zeta_x\|\\
&\quad+\|\zeta_x\|^{\frac{1}{2}}\|\zeta_{xx}\|^{\frac{1}{2}}\|H^2_x\|_{L^1}\Big\}ds\\
&\le C\chi\epsilon_\delta+ \int_0^t\Big\{\beta\big(\|\psi_{xx}\|^2+\|\zeta_{xx}\|^2\big)+C_\beta\big(\|\psi_{x}\|^{\frac{2}{3}}\|G^2_x\|_{L^1}^{\frac{4}{3}}+\|\zeta_{x}\|^{\frac{2}{3}}\|H^2_x\|_{L^1}^{\frac{4}{3}}\big)\Big\}ds\\
&\le \int_0^t\beta\big(\|\psi_{xx}\|^2+\|\zeta_{xx}\|^2\big)ds+(C_\beta+C)\chi\epsilon_\delta.
\end{aligned}
\end{align}
Then choosing $\beta$ suitably small, we derive directly Lemma \ref{lemm7} 
from \eqref{eq82}-\eqref{eq85}. 
\end{proof}

After proving the Proposition $\ref{pro15}$, we can extend the unique local solution $(\phi,\psi,\zeta)(x,t)$ in Proposition $\ref{pro151}$
to $T=\infty$.
By Proposition $\ref{pro15}$ and equations \eqref{equ20}, we have
$$\int_0^\infty\Big(\|(\phi_x,\psi_x,\zeta_x)\|^2+\big|\frac{d}{dt}\|(\phi_x,\psi_x,\zeta_x)(t)\|^2\big|\Big)dt<\infty,$$ therefore,
$$ \lim_{t\to +\infty} \|(\phi_x,\psi_x,\zeta_x)\|^2=0,$$
which together with \eqref{eq91} and Sobolev inequality imply \eqref{eq030}.
Thus, 
 the proof of the main Theorem \ref{thm1} is completed.


\vspace{1.5cm}

\end{document}